\documentclass[final,1p,times]{elsarticle}

\usepackage{amsmath,amssymb,amsfonts,latexsym,enumerate}
\usepackage[T1]{fontenc}
\usepackage[cp1250]{inputenc}

\newcommand{\E}{\mbox{$\mathbb{E}$}}
\newcommand{\nat}{\mbox{$\mathbb{N}$}}
\newcommand{\real}{\mbox{$\mathbb{R}$}}

\newcommand{\Bcl}{\mbox{$\mathcal{B}$}}

\newtheorem{theorem}{Theorem}[section]
 \newtheorem{corollary}[theorem]{Corollary}
 \newtheorem{lemma}[theorem]{Lemma}
 \newtheorem{proposition}[theorem]{Proposition}

 \newdefinition{remark}[theorem]{Remark}
 \newdefinition{example}[theorem]{Example}
\newdefinition{definition}[theorem]{Definition}

 \newproof{proof}{Proof}
 \newproof{pfs}{Proof of Theorem~\ref{th:support}}
 \newproof{pfd}{Proof of Lemma~\ref{lm:delta}}

\newcommand{\leqcx}{\leqslant_{\textit{cx}}}

\newcommand{\Lob}[1]{\mathcal{L}_{#1}}

\newcommand{\pd}{\sqrt{2}}
\newcommand{\pp}{\sqrt{5}}

\numberwithin{equation}{section}

\begin{document}

\begin{frontmatter}



\title{On a generalization of a theorem of Levin and Ste\v{c}kin and inequalities of the Hermite-Hadamard type}


\author{Teresa Rajba}
\ead{trajba@ath.bielsko.pl}
\address{Department of Mathematics, University of Bielsko-Biala, ul. Willowa 2, 43-309 Bielsko-Biala, Poland}

\begin{abstract}
We give new necessary and sufficient conditions for higher order convex ordering. These results generalize the Levin-Ste\v{c}kin theorem (1960) on convex ordering. The obtained results can be useful in the study of the Hermite-Hadamard type inequalities and in particular inequalities between the quadrature operators.

\end{abstract}


\end{frontmatter}

% \linenumbers

\section{Introduction and preliminaries}
 In this paper we give new criteria for the verification of higher order convex orders. These criteria can be used to prove the Hermite-Hadamard type inequalities for higher order convex functions.

Let $f \colon [a,b]\to\real$ be a convex function. The following double inequality
\begin{equation}
f\left(\frac{a+b}2\right)\leqslant \frac1{b-a}\int_a^b f(x)\,dx \leqslant \frac{f(a)+f(b)}2 \label{eq:oli0}
\end{equation}
is known as the Hermite-Hadamard inequality for convex functions (see  \cite{DraPe2000} for many generalizations and applications of \eqref{eq:oli0}).

In many papers (see, for example, \cite {Florea, Rajba12f, Rajba14, Rajba12g, Szostok2014a, TSzostok2014, OlbrysSzostok2014}) are studied the Hermite-Hadamard type inequalities based on the convex stochastic ordering properties. In the paper \cite{Rajba12f}, to get a simple proof of some known Hermite-Hadamard type inequalities as well as to obtaining new Hermite-Hadamard type inequalities, is used the Ohlin lemma  on sufficient conditions for convex stochastic ordering. Recently, the Ohlin lemma is also used  to study the inequalities of the Hermite-Hadamard type in \cite {Rajba14, Rajba12g, Szostok2014a, TSzostok2014, OlbrysSzostok2014}. In the papers \cite {Szostok2014a, TSzostok2014, OlbrysSzostok2014}, furthermore, to examine the Hermite-Hadamard type inequalities is used the Levin-Ste\v{c}kin theorem \cite {LevinSteckin1960} (see also \cite {NicPer06}), which gives necessary and sufficient conditions for the stochastic convex ordering.

Let us recall some basic notions and results on the stochastic convex order (see, for example, \cite {DenLefSha98}). As usual, $F_X$ denotes the distribution function of a random variable $X$ and $\mu_X$ is the distribution corresponding to $X$. For real valued random variables $X,Y$ with a finite expectation, we say that $X$ is dominated by $Y$ in \textit{convex ordering} sense if
 $$\E f(X) \leqslant \E f(Y)$$
for all convex functions $f \colon \real \to \real$ (for which the expectations exist). In that case we write $X \leqcx Y$, or $\mu_X \leqcx \mu_Y$.

In the following Ohlin lemma \cite{Ohlin69} are given sufficient conditions for convex stochastic ordering.
\begin{lemma}[\cite{Ohlin69}]\label{lemma:3}
Let $X,Y$ be two random variables such that $\E X=\E Y$. If the distributions functions $F_X, F_Y$ cross exactly one time, i.e., for some $x_0$ holds
\begin{equation}\label{eq:ls1.2}
F_X(x) \leqslant F_Y(x) \textit{ if } x < x_0 \textit{ and } F_X(x) \geqslant F_Y(x) \textit{ if } x > x_0,
\end{equation}
then 
\begin{equation}\label{eq:levin1}
\E f(X) \leqslant \E f(Y)
\end{equation}
for all convex functions $f \colon \real \to \real$.
\end{lemma}
\begin{remark}
As noticed Szostok in \cite{Szostok2014a}, if the measures $\mu_X $, $\mu_Y$ corresponding to $X$ and $Y$, respectively, are concentrated on the interval $[a,b]$ then, in fact, inequality \eqref{eq:levin1} is satisfied for all continuous convex functions $f \colon [a,b] \to \real$ .
\end{remark}
\begin{remark}
The inequality \eqref{eq:oli0} may be easily proved with the use of the Ohlin lemma (see\cite{Rajba12f}). Indeed, let $X$, $Y$, $Z$ be three random variables with the distributions $\mu_X = \delta_{(a+b)/2}$, $\mu_Y$ which is equally distributed in $[a,b]$ and $\mu_Z = \frac{1}{2}(\delta_a+\delta_b)$, respectively. Then it is easy to see that the pairs $(X,Y)$ and $(Y,Z)$ satisfy the assumptions of the Ohlin lemma, and using \eqref{eq:levin1}, we obtain \eqref{eq:oli0}.
\end{remark}
As we can see, the Ohlin lemma is a strong tool, however, it is worth noticing that in the case of some inequalities, the distribution functions cross more than once. Therefore a simple application of the Ohlin lemma is impossible and some additional idea is needed.

In the papers \cite{TSzostok2014, OlbrysSzostok2014}, the authors used the Levin-Ste\v{c}kin theorem  \cite{LevinSteckin1960} (see also \cite{NicPer06}, Theorem 4.2.7).
\begin{theorem}[\cite{LevinSteckin1960}]\label{th:1.3}
Let $a,b\in\real$, $a<b$ and let $F_1,F_2 \colon [a,b]\to\real$ be functions with bounded variation such that $F_1(a)=F_2(a)$. Then, 
in order that
\begin{equation*}
\int_a^bf(x)dF_1(x)\leqslant\int_a^bf(x)dF_2(x),
\end{equation*}
for all continuous convex functions $f \colon [a,b]\to\real,$
it is necessary and sufficient that $F_1$ and $F_2$ verify the following three conditions:
\begin{eqnarray}
F_1(b)& = & F_2(b)\label{eq:ls1.4},\\
\int_a^b F_1(x)dx & = & \int_a^b F_2(x)dx, \label{eq:ls1.6}\\
  \int_a^x F_1(t)dt & \leqslant & \int_a^x F_2(t)dt\quad for\: all\:x\in(a,b).\label{eq:ls1.5} 
\end{eqnarray}
\end{theorem}
\begin{remark}
Observe, that if the measures $\mu_X$, $\mu_Y$, corresponding to the random variables $X$, $Y$, respectively, occurring in the Ohlin lemma, are concentrated on the interval $[a,b]$, the Ohlin lemma is an easy consequence of Theorem \ref{th:1.3}. Indeed, $\mu_X$ and $\mu_Y$ are probabilistic measures, thus we have $F_X(a)=F_Y(a)=0$ and $F_X(b)=F_Y(b)=1$. Moreover, $\E X=\E Y$ yields \eqref{eq:ls1.6} and from the inequalities \eqref{eq:ls1.2} we obtain \eqref{eq:ls1.5}.

\end{remark}

Szostok \cite{TSzostok2014} used Theorem \ref{th:1.3} to make an observation, which is more general than Ohlin's lemma and concerns the situation when the functions $F_1$ and $F_2$ have more crossing points than one. First we need the following definitions.

Define the number of sign changes of a function $\varphi \colon \real \to \real$ by
$$
S^-(\varphi) = \sup \{S^-[\varphi(x_1),\varphi(x_2),\ldots,\varphi(x_k)]\colon x_1<x_2<\ldots x_k\in\real,\: k\in\nat\},
$$
where $S^-[y_1,y_2,\ldots,y_k]$ denotes the number of sign changes in the sequence $y_1,y_2,\ldots,y_k$ (zero terms are being discarded). Two real functions $\varphi_1,\varphi_2$ are said to have $n$ crossing points (or cross each other $n$-times) if $S^-(\varphi_1-\varphi_2)=n$. Let $a=x_0<x_1< \ldots<x_n<x_{n+1}=b$. We say that the functions  $\varphi_1,\varphi_2$ crosses $n$-times at the points $x_1,x_2,\ldots,, x_n$ (or that $x_1,x_2,\ldots,, x_n$ are the points of sign changes of $\varphi_1-\varphi_2$) if $S^-(\varphi_1-\varphi_2)=n$ and there exist $a<\xi_1<x_1< \ldots<\xi_n<x_n<\xi_{n+1}<b$ such that $S^-[\xi_1,\xi_2,\ldots,\xi_{n+1}]=n$.

Then, Lemma 2 given in \cite{TSzostok2014} can be rewritten in the following form.
\begin{lemma}[\cite{TSzostok2014}]\label{lemma:1.5}
Let $a,b\in\real$, $a<b$ and let $F_1,F_2 \colon (a,b)\to\real$ be functions with bounded variation such that $F_1(a)=F_2(a)$, $F_1(b)=F_2(b)$, $F=F_2-F_1$, $\int_a^b F(x)dx=0$. Let $a<x_1< \ldots<x_m<b$ be the points of sign changes of $F$ and $F(t)\geq 0$   for $t\in(a,x_1)$.  
\begin{itemize}
\item If $m$ is even then 
%$F_1\leqcx F_2$
the inequality 
\begin{equation}\label{eq:ord1}
\int_a^bf(x)dF_1(x)\leqslant\int_a^bf(x)dF_2(x)
\end{equation}
is not satisfied by all continuous convex functions $f \colon [a,b]\to\real.$ 

\item If $m$ is odd, define $A_i$ ($i=0,1,\ldots,m$, $x_0=a$, $x_{m+1}=b$)
$$
A_i = \int_{x_i}^{x_{}i+1} |F(x)|dx.
$$
Then the inequality \eqref{eq:ord1} is satisfied for all continuous convex functions $f \colon [a,b]\to\real,$  if and only if the following inequalities hold true:
\begin{equation}\label{eq:a1}
\begin{split}
A_0 & \geqslant  A_1, \\
A_0+A_2 & \geqslant  A_1+A_3 ,\\
& \vdots \\
A_0+A_2+ \ldots+A_{m-3} & \geqslant  A_1+A_3+ \ldots+A_{m-2}.
\end{split}
\end{equation}
\end{itemize}
\end{lemma}
\begin{remark}\label{rmk:a2}
Let
\begin{equation*}
H(x)=\int_a^x F(t)dt.
\end{equation*}
Then the inequalities  \eqref{eq:a1}
are equivalent to the following inequalities
\begin{equation*}
H(x_2)\geqslant  0,\: H(x_4)\geqslant  0,\:H(x_6)\geqslant  0, \:\ldots, \:H(x_{m-1})\geqslant 0.
\end{equation*}
\end{remark}

Now we are going to study Hermite-Hadamard type inequalities for higher-order convex functions. Many results on higher order generalizations of the Hermite-Hadamard type inequality one can found, among others, in \cite{Bes08, BesPal02, BesPal03, BesPal04, BesPal10, DraPe2000, Rajba12f, Rajba12g}. In recent papers \cite{Rajba12f, Rajba12g} the theorem of M. Denuit, C.Lefevre and  M. Shaked \cite{DenLefSha98} on sufficient conditions for $s$-convex ordering was used to prove Hermite-Hadamard type inequalities for higher-order convex functions.

Let us review some notations. The convexity of $n$-th order (or $n$-convexity) was defined in terms of divided differences by Popoviciu \cite{Popoviciu1934}, however, we will not state it have. Instead we list some properties of $n$-th order convexity which are equivalent to Popoviciu's definition (see \cite{Kuczma1985}).

\begin{proposition}\label{prop:16}
A function $f \colon (a,b)\to\real$ is $n$-convex on $(a,b)$ $(n \geqslant 1)$ if and only if its derivative $f^{(n-1)}$ exists and is convex on $(a,b)$ (with the convention $f^{(0)}(x) = f(x)$).
\end{proposition}

\begin{proposition}\label{prop:17}
Assume that $f \colon [a,b] \to\real$ is $(n+1)$-times differentiable on $(a,b)$ and continuous on $[a,b]$ ($n \geqslant 1$). Then $f$ is $n$-convex if and only if $f^{(n+1)}(x)\geqslant 0$, $x \in (a,b)$.
\end{proposition}

For real valued random variables $X,Y$ and any integer $s \geqslant 2$ we say that $X$ is dominated by $Y$ in $s$-\textit{convex ordering} sense if $\E f(X) \leqslant \E f(Y)$ for all $(s-1)$-convex functions $f \colon \real \to \real$, for which the expectations exist (\cite{DenLefSha98}). In that case we write $X \leqslant_{s-cx} Y$, or $\mu_X \leqslant_{s-cx} \mu_Y$, or $F_X \leqslant_{s-cx} F_Y$. Then the order $\leqslant_{2-cx}$ is just the usual convex order $\leqslant_{cx}$.

A very useful criterion for the verification of the $s$-convex order is given by Denuit, Lef\`{e}vre and Shaked in \cite{DenLefSha98}.
\begin{proposition}[\cite{DenLefSha98}]\label{prop:18}
Let $X$ and $Y$ be two random variables such that $\E (X^j-Y^j)=0$, $j=1,2,\ldots,s-1$ ($s \geqslant 2$). If $S^-(F_X-F_Y)=s-1$ and the last sign of $F_X-F_Y$ is positive, then $X \leqslant_{s-cx} Y$.
\end{proposition}
Proposition \ref{prop:18} can be rewritten in the following form.
\begin{proposition}[\cite{DenLefSha98}]\label{prop:19}
Let $X$ and $Y$ be two random variables such that 
\begin{equation*}
%\label{eq:ord3}
\E (X^j-Y^j)=0,\quad j=1,2,\ldots,s \:(s \geqslant 1).
\end{equation*}
If the distribution functions $F_X$ and $F_Y$ cross exactly $s$-times at points $x_1<x_2< \ldots <x_s$ and 
\begin{equation*}
%\label{eq:ord4}
(-1)^{s+1}\left(F_Y(x)-F_X(x)\right)\geqslant 0\quad for\: all\: x\leqslant x_1
\end{equation*}
then 
\begin{equation}\label{eq:ord5}
\E f(X) \leqslant \E f(Y)
\end{equation}
for all $s$-convex functions $f \colon \real \to \real$.
\end{proposition}
\begin{remark}
Observe, that if the measures $\mu_X$, $\mu_Y$, corresponding to the random variables $X$, $Y$, respectively, occurring in Proposition \ref{prop:19}, are concentrated on the interval $[a,b]$, then in fact inequality \eqref{eq:ord5} is satisfied for all continuous $s$-convex functions $f \colon [a,b] \to \real$.
\end{remark}

Proposition \ref{prop:19} is a counterpart of the Ohlin lemma concerning convex ordering.  This proposition  gives sufficient conditions for $s$-convex ordering, and is very useful for the verification of higher order convex orders,  however, it is worth noticing that in the case of some inequalities, the distribution functions cross more than $s$-times. Therefore a simple application of this proposition is impossible and some additional idea is needed. 
%\textbf{Znane sa warunki konieczne i wystarczajace  dla s-convex orders, ale nie sa one proste do zastosowania (see for example \cite{DenLefSha98}.}

In this paper we give a theorem on necessary and sufficient conditions for higher order convex stochastic ordering, which can be useful in the study of Hermite-Hadamard type inequalities for higher order convex functions, and in particular inequalities between the quadrature operators. This theorem is a counterpart of the Levin-Ste\v{c}kin theorem \cite{LevinSteckin1960} concerning convex stochastic ordering. The necessary and sufficient conditions, which we give in this paper, concern higher order stochastic convex ordering of signed measures, and are generalization of results given by Denuit, Lef\`{e}vre and Shaked in  \cite{DenLefSha98}. Moreover, our criteria can be easier to checking of higher order convex orders, than those given in \cite{DenLefSha98}.
%%%%%%%%%%%%%%%%%%%%%%%%%%%%%%%%%%%%%%%%%%%%%%%%%%%%%%%%%%%%%%%%%%%%%%%%%%%%%%%%%
\section{Main results}
Let  $F_1,F_2 \colon \real\to\real$ be two functions with bounded variation and $\mu_1$, $\mu_2$ be the signed measures corresponding to $F_1$, $F_2 $, respectively. We say that $F_1$ is dominated by $F_2 $ in $(n+1)$-\textit{convex ordering sense} $(n\geqslant 1)$ if 
\begin{equation*}
%\label{eq:lsord}
\int_{-\infty}^{\infty}f(x)dF_1(x)\leqslant\int_{-\infty}^{\infty}f(x)dF_2(x),
\end{equation*}
for all $n$-convex functions $f \colon \real\to\real$.
In that case we write $F_1 \leqslant_{(n+1)-cx} F_2$, or $\mu_1 \leqslant_{(n+1)-cx} \mu_2$. 

In the following theorem we give necessary and sufficient conditions for $(n+1)$-convex ordering of two functions with bounded variation.
\begin{theorem}\label{th:2.1}
Let $a,b\in\real$, $a<b$, $n\in \nat$ and let $F_1,F_2 \colon [a,b]\to\real$ be two functions with bounded variation such that $F_1(a)=F_2(a)$. Then, 
in order that
\begin{equation*}
%\label{eq:lsord}
\int_a^bf(x)dF_1(x)\leqslant\int_a^bf(x)dF_2(x),
\end{equation*}
for all continuous $n$-convex functions $f \colon [a,b]\to\real,$
it is necessary and sufficient that $F_1$ and $F_2$ verify the following  conditions:
\begin{equation}\label{eq:ls2.2}
F_1(b)= F_2(b),
\end{equation}
\begin{equation}\label{eq:ls2.3}
\int_a^b F_1(x)dx = \int_a^b F_2(x)dx, 
\end{equation}
\begin{multline}\label{eq:ls2.4} 
\int_a^b \int_a^{x_{k-1}} \ldots \int_a^{x_1} F_1(t)dtdx_1\ldots dx_{k-1}=\\ 
 \int_a^b \int_a^{x_{k-1}} \ldots \int_a^{x_1} F_2(t)dtdx_1\ldots dx_{k-1},\quad
 for\ \  k=2,\ldots ,n,
 \end{multline}
  \begin{multline}\label{eq:ls2.5} 
(-1)^{n+1}\int_a^x \int_a^{x_{n-1}} \ldots \int_a^{x_1} F_1(t)dtdx_1\ldots dx_{n-1}\leqslant\\ 
 (-1)^{n+1}\int_a^x \int_a^{x_{n-1}} \ldots \int_a^{x_1} F_2(t)dtdx_1\ldots dx_{n-1}, \ \ for \ any \ x\in(a,b).
 \end{multline}
\end{theorem}

First we prove the following lemma.
\begin{lemma}\label{lemma:2.2}
Let $F \colon [a,b]\to\real$ be a function with bounded variation. Let $f \colon [a,b]\to\real$ be an $n$-convex function of the class $C^{n+1}$ on $(a,b)$. Then
\begin{equation}\label{eq:2.6}
\int_a^bf(x)dF(x)=\biggl[F(x)f(x)\biggl]_{x=a}^{x=b}-\int_a^bF(x)f'(x)dx,
\end{equation}
\begin{equation}\label{eq:2.7}
\int_a^bf(x)dF(x)=\biggl[F(x)f(x)\biggl]_{x=a}^{x=b}-\biggl[\int_a^xF(t)dtf'(x)\biggl]_{x=a}^{x=b}+\int_a^b\int_a^xF(t)dtf''(x)dx,
\end{equation}
\begin{multline}\label{eq:2.8} 
\int_a^bf(x)dF(x)=\biggl[F(x)f(x)\biggl]_{x=a}^{x=b}-\biggl[\int_a^xF(t)dtf'(x)\biggl]_{x=a}^{x=b}+\biggl[\int_a^x\int_a^{x_1}F(t)dtdx_1f''(x)\biggl]_{x=a}^{x=b}+\ldots \\
\ldots +\biggl[(-1)^k\int_a^x\int_a^{x_{k-1}}\ldots  \int_a^{x_1}F(t)dtdx_1\ldots dx_{k-1}f^{(k)}(x)\biggl]_{x=a}^{x=b}+\\
+(-1)^{k+1}\int_a^b\int_a^x\int_a^{x_{k-1}}\ldots  \int_a^{x_1}F(t)dtdx_1\ldots dx_{k-1}f^{(k+1)}(x)dx,\quad for \:k=2,\ldots,n.
 \end{multline}
\end{lemma}
\begin{proof}
The proof is by induction. Integrating by parts and using the equalities $F(x)=\left( \int_a^xF(t)dt \right)'$ and $\int_a^xF(t)dt=\left(\int_a^x\int_a^{x_1}F(t)dtdx_1 \right)'$, we obtain immediately \eqref{eq:2.6}, \eqref{eq:2.7} and \eqref{eq:2.8} for $k=2$.

Put
$$
I_k(x)=\int_a^x\int_a^{x_{k}}\ldots  \int_a^{x_1}F(t)dtdx_1\ldots dx_{k},\quad for \:x\in(a,b), \:k=1,2,\ldots,n.
$$
Then we have
\begin{equation}\label{eq:2.9}
I_{k-1}(x)=\left(I_k(x)\right)',\quad for \:x\in(a,b), \:k=1,2,\ldots,n.
\end{equation}
Assume that \eqref{eq:2.8} holds for some $k=2,\ldots,n-1$. Integrating by parts and using \eqref{eq:2.9}, we obtain that the last summand in \eqref{eq:2.8} can be rewritten in the form
\begin{multline*}
(-1)^{k+1}\int_a^b\int_a^x\int_a^{x_{k-1}}\ldots  \int_a^{x_1}F(t)dtdx_1\ldots dx_{k-1}f^{(k+1)}(x)dx=\\=(-1)^{k+1}\int_a^bI_{k-1}(x)f^{(k+1)}(x)dx 
=\biggl[I_k(x)f^{(k+1)}(x)\biggl]_{x=a}^{x=b}+(-1)^{k+2}\int_a^bI_{k}(x)f^{(k+2)}(x)dx, 
 \end{multline*}
which implies that \eqref{eq:2.8} holds for $k+1$. Thus \eqref{eq:2.8} holds for all $k=2,\ldots,n$. The lemma is proved.
\end{proof}
The proof of Theorem \ref{th:2.1} is an immediate consequence of the following lemma.
\begin{lemma}\label{lemma:2.3}
Let $F \colon [a,b]\to\real$ be a function with bounded variation such that $F(a)=0$. Then in order that  
\begin{equation}\label{eq:2.10}
\int_a^bf(x)dF(x)\geqslant 0,
\end{equation}
for all continuous $n$-convex functions $f \colon [a,b]\to\real,$
it is necessary and sufficient that $F$ satisfies the following  conditions:
\begin{align} 
F(b)&=0, \label{eq:2.11}\\ 
\int_a^b F(x)dx & =0,\label{eq:2.12}\\
\int_a^b \int_a^{x_{k-1}} \ldots \int_a^{x_1} F(t)dtdx_1\ldots dx_{k-1}&=0,\quad
 for\ \  k=2,\ldots ,n,\label{eq:2.13}\\
 (-1)^{n+1}\int_a^x \int_a^{x_{n-1}} \ldots \int_a^{x_1} F(t)dtdx_1\ldots dx_{n-1}&\geqslant 0,\quad for \ any \ x\in(a,b). \label{eq:2.14}
 \end{align}
\end{lemma}
\begin{proof}
Via an approximation argument we may restrict to the case when $f$ is of the class $C^{n+1}((a,b))$. 

We now prove the sufficiency. By Lemma \ref{lemma:2.2}, using \eqref{eq:2.7} and \eqref{eq:2.8} with $k=n$, and taking into account \eqref{eq:2.11}-\eqref{eq:2.13} we get
\begin{equation}\label{eq:2.15}
\int_a^bf(x)dF(x)=(-1)^{n+1}\int_a^b\int_a^x \int_a^{x_{n-1}} \ldots \int_a^{x_1} F(t)dtdx_1\ldots dx_{n-1}f^{(n+1)}(x)dx.
\end{equation}
Then, by \eqref{eq:2.14} and Proposition \ref{prop:17}, we obtain \eqref{eq:2.10}.

We now prove the necessity. The necessity of \eqref{eq:2.11} follows by checking our statement for $f=1$ and $f=-1$.

 The necessity of \eqref{eq:2.12} follows by checking our statement for $f(x)=x$ and $f(x)=-x$ and by using \eqref{eq:2.11}, \eqref{eq:2.6}.
 
 The necessity of \eqref{eq:2.13} we prove by induction on $k$. The necessity of \eqref{eq:2.13} for $k=2$ follows by checking our statement for $f(x)=x^2$ and $f(x)=-x^2$, using \eqref{eq:2.7} and taking into account \eqref{eq:2.11}, \eqref{eq:2.12}. Assume, that the equality 
 \begin{equation}\label{eq:2.16}
 \int_a^b \int_a^{x_{l-1}} \ldots \int_a^{x_1} F(t)dtdx_1\ldots dx_{l-1}=0
 \end{equation}
 holds , for some $k=2,\ldots ,n-1$ and all $l=2,\ldots ,k$. Then we check our statement for $f(x)=x^{k+1}$ and $f(x)=-x^{k+1}$. Using \eqref{eq:2.8} and taking into account \eqref{eq:2.11}, \eqref{eq:2.12} and \eqref{eq:2.16} for $l=2,\ldots ,k$, we obtain \eqref{eq:2.16} for $l=k+1$. Consequently, we obtain that \eqref{eq:2.13} is satisfied for all $k=2,\ldots ,n$.
 
 By \eqref{eq:2.8} with $k=n$ and taking into account \eqref{eq:2.11}-\eqref{eq:2.13}, we obtain that \eqref{eq:2.15} holds. Then, for the necessity of \eqref{eq:2.14}, notice that $(-1)^{n+1}\int_a^x \int_a^{x_{n-1}} \ldots \int_a^{x_1} F(t)dtdx_1\ldots dx_{n-1}<0$ for some $x\in(a,b)$, yields an interval $I$ around $x$  on which this expression is still negative. Choosing $f$ such that $f^{(n+1)}=0$ outside $I$, the equality \eqref{eq:2.15} leads to a contradiction. Thus \eqref{eq:2.14} is satisfied. The lemma is proved.
 \end{proof}
 \begin{corollary}\label{cor:2.4}
Let $\mu_1$, $\mu_2$ be two signed measures on $\Bcl(\real)$, which are concentrated on $(a,b)$, and such that $\int_a^b\lvert x \rvert ^n \mu_i(dx)<\infty$, $i=1,2$. Then in order that 
\begin{equation*}
\int_a^bf(x)d\mu_1(x)\leqslant\int_a^bf(x)d\mu_2(x),
\end{equation*}
for continuous $n$-convex functions $f \colon [a,b]\to\real$, it is necessary and sufficient that  $\mu_1$, $\mu_2$ verify the following conditions:
\begin{align}
\mu_1\left((a,b)\right)&=\mu_1\left((a,b)\right),\label{eq:2.17}\\
\int_a^b x  ^k \mu_1(dx)&=\int_a^b x  ^k \mu_2(dx),\quad for \:k=1,\ldots ,n,\label{eq:2.18}\\
\int_a^b \bigl(t-x\bigr)  ^n _+\mu_1(dt)&=\int_a^b \bigl(t-x\bigr) ^n _+\mu_2(dt),\quad for \:x\in (a,b),\label{eq:2.19}
\end{align}
where $y^n_+=\Bigl\lbrace max \lbrace y,0\rbrace \Bigr\rbrace ^n$, $y\in \real$.
\end{corollary}
\begin{proof}
Let $F_1$, $F_2$ be the distribution functions corresponding to $\mu_1$, $\mu_2$, respectively. Then $\mu_i(dt)=dF_i(t)$, $i=1,2$. Since $\mu_1$ and $\mu_2$ are concentrated on $(a,b)$, we have $F_1(a)=F_2(a)$. That \eqref{eq:ls2.2} and \eqref{eq:2.17} are equivalent is obvious. Put $F=F_2-F_1$. By \eqref{eq:2.6} with $f(x)=x$, and taking into account \eqref{eq:ls2.2}, it follows that the conditions \eqref{eq:ls2.3} and \eqref{eq:2.18} for $k=1$ are equivalent. The equivalence of \eqref{eq:ls2.4} and \eqref{eq:2.18} for $k=2,\ldots ,n$, can be proved, using \eqref{eq:2.7} and \eqref{eq:2.18}, by induction on $k$. We omit the proof.

Next, by reversing the order of integration in \eqref{eq:ls2.5}, we obtain
\begin{align*}
&(-1)^{n+1}\int_a^x \int_a^{x_{n-1}} \ldots \int_a^{x_1} (F_2(t)-F_1(t))dtdx_1\ldots dx_{n-1}\\
&=(-1)^{n+1}\int_a^x \int_a^{x_{n-1}} \ldots \int_a^{x_1} F(t)dtdx_1\ldots dx_{n-1}\\
&=(-1)^{n+1}\int_a^x \dfrac{(x-t)^n}{n!}dF(t)=(-1)^{n+1}(-1)^{n}\int_a^x \dfrac{(t-x)^n}{n!}dF(t)=-\int_a^x \dfrac{(t-x)^n}{n!}dF(t)\\
&=\int_a^b \dfrac{(t-x)^n}{n!}dF(t)-\int_a^x \dfrac{(t-x)^n}{n!}dF(t)=\int_x^b \dfrac{(t-x)^n}{n!}dF(t)=\int_a^b \dfrac{(t-x)^n_+}{n!}dF(t),
\end{align*}
which implies the equivalence of \eqref{eq:ls2.5} and \eqref{eq:2.19}. The corollary is proved.
\end{proof}

Note that Theorem \ref{th:2.1} can be rewritten in the following form.
\begin{theorem}\label{th:2.7a}
Let $F_1,F_2 \colon [a,b]\to\real$ be two functions with bounded variation such that $F_1(a)=F_2(a)$. Let
\begin{align*}
H_0(t_0)&=F_2(t_0)-F_1(t_0), \quad for\: t_0\in [a,b],\\
H_{k}(t_k)&=\int_a^{t_{k-1}}H_{k-1}(t_{k-1})dt_{k-1},\quad for\: t_{k}\in [a,b], \  k=1,2,\ldots ,n.
\end{align*}
Then, 
in order that
\begin{equation*}
\int_a^bf(x)dF_1(x)\leqslant\int_a^bf(x)dF_2(x),
\end{equation*}
for all continuous $n$-convex functions $f \colon [a,b]\to\real,$
it is necessary and sufficient that the following  conditions are satisfied:
\begin{align*}
H_k(b)&=0, \quad for\: k=0,1,2,\ldots ,n,\\
(-1)^{n+1}H_{n}(x)&\geqslant 0 ,\quad for\:all\: x\in (a,b).
\end{align*}
\end{theorem}
\begin{remark}\label{rmk:2.8a}
The functions $H_1, \ldots,H_n$, that appear in Theorem \ref{th:2.7a} can be obtained from the following formulas
	\begin{equation}\label{eq:2.8c}
H_{n}(x)=(-1)^{n+1}\int_a^b \frac{(t-x)_+^n}{n!}d (F_2(t)-F_1(t)), 
\end{equation} 
\begin{equation}\label{eq:2.8d}
H_{k-1}(x)=H_{k}^{\,'}(x), \quad k=2,3,\ldots ,n.
\end{equation}
\end{remark}

Note that the function $(-1)^{n+1}H_{n-1}$ that appears in Theorem \ref{th:2.7a} play a role similar to the role of the function $F=F_2-F_1$ in Lemma \ref{lemma:1.5}. Consequently, from Theorem \ref{th:2.7a}, Lemma \ref{lemma:1.5} and Remarks \ref{rmk:a2}, \ref{rmk:2.8a} we obtain immediately the following useful criterion for the verification of higher order convex ordering.
\begin{corollary}\label{cor:2.7b}
Let $F_1,F_2 \colon [a,b]\to\real$ be functions with bounded variation such that $F_1(a)=F_2(a)$, $F_1(b)=F_2(b)$ and $H_{k}(b)=0$ $(k=1,2,\ldots ,n)$, where $H_k(x)$ $(k=1,2,\ldots ,n)$ are given by \eqref{eq:2.8c} and \eqref{eq:2.8d}. Let $a<x_1< \ldots<x_m<b$ be the points of sign changes of the function $H_{n-1}$ and let $(-1)^{n+1}H_{n-1}(x)\geqslant 0$   for $x\in(a,x_1)$.
\begin{itemize}
\item If $m$ is even then the inequality
 \begin{equation}\label{eq:2.7c}
\int_a^bf(x)dF_1(x)\leqslant\int_a^bf(x)dF_2(x),
\end{equation}
is not satisfied by all continuous $n$-convex functions $f \colon [a,b]\to\real$.
 
\item If $m$ is odd, then the inequality \eqref{eq:2.7c} is satisfied for all continuous $n$-convex functions $f \colon [a,b]\to\real$ if and only if 
%the followinf inequalities are satisfied
\begin{equation}\label{eq:2.7d}
(-1)^{n+1}H_n(x_2)\geq 0, (-1)^{n+1}H_n(x_4)\geq 0,\ldots,
 (-1)^{n+1}H_n(x_{m-1})\geq 0.
\end{equation}
\end{itemize}
\end{corollary}
In \cite{DenLefSha98} can be found the following necessary and sufficient conditions for the verification of the $(s+1)$-convex order.
\begin{proposition}[\cite{DenLefSha98}]\label{prop:2.5}
If $X$ and $Y$ are two real valued random variables such that $\E \lvert X \rvert^{s}<\infty$ and  $\E \lvert Y\rvert^{s}<\infty$ then 
%$X \leqslant_{s-cx} Y$
\begin{equation}\label{eq:2.5c}
\E f(X) \leqslant \E f(Y)
\end{equation}
for all continuous $s$-convex functions $f \colon \real \to \real$
if and only if 
\begin{align}
\E X^k&=\E Y^k, \quad for \:k=1,2,\ldots,s,\label{eq:2.5a} \\
\E (X-t)^{s}_+&\leqslant \E (Y-t)^{s}_+, \quad for \:all\: t \in \real\label{eq:2.5b}.
\end{align}
\end{proposition}
\begin{remark}
The inequality \eqref{eq:2.5b} coincides with \eqref{eq:2.5c} for the spline function $f(x)=(x-t)^{s}_+$. Moreover, it is well known that $s$-convex function has the integral representation, such that the spline functions are the generating functions (see \cite{Rajba2011}).
\end{remark}
\begin{remark}
Note, that if the measures $\mu_X$, $\mu_Y$, corresponding to the random variables $X$, $Y$, respectively, occurring in Proposition \ref{prop:2.5}, are concentrated on some interval $[a,b]$, then this proposition is an easy consequence of Corollary \ref{cor:2.4}.
\end{remark}
\begin{remark}
Necessary and sufficient conditions for the verification of the $(s+1)$-convex order, which are given in Proposition \ref{prop:2.5}, can be difficult to checking. In this paper, we give conditions that can be easier for verification of higher order convex orders. Moreover, our conditions concern not only probabilistic measures but also signed measures. 
\end{remark}

In the numerical analysis are studied some inequalities, which are connected with quadrature operators. These inequalities, called extremalities, are a particular case of the Hermite-Hadamard type inequalities. Many extremalities are known in the numerical analysis (cf. \cite{Bes08}, \cite{BraSch81}, \cite{BraPet} and the references therein).
The numerical analysts prove them using the suitable differentiability assumptions. As proved W\k{a}sowicz in the papers \cite{SzWas07b}, \cite{SzWas08}, \cite{SzWas10}, for convex functions of higher order some extremalities can be obtained without assumptions of this kind, using only the higher order convexity itself. The support-type properties play here the crucial role. As we show in \cite{Rajba12f, Rajba12g}, some extremalities can be proved using a probabilistic characterization.The extremalities, which we study are known, however our method using the Ohlin lemma \cite{Ohlin69} and the Denuit-Lef\`{e}vre-Shaked theorem \cite{DenLefSha98} on sufficient conditions for the convex stochastic ordering seems to be quite easy. It is worth noting that, these theorems do not apply to proving some extremalities (see \cite{Rajba12f, Rajba12g}). In these cases can be useful results given in this paper.

For a function $f:[-1,1]\to\real$ we consider two operators 
\begin{eqnarray*}
 C(f)&:=&\tfrac{1}{3}\Bigl(
                      f\bigl(-\tfrac{\pd}{2}\bigr)+f(0)+
					  f\bigl(\tfrac{\pd}{2}\bigr)
                     \Bigr),\\\\
										\Lob{4}(f)&:=&\tfrac{1}{12}\Bigl(f(-1)+f(1)\Bigr)
              +\tfrac{5}{12}\Bigl(f\bigl(-\tfrac{\pp}{5}\bigr)
              +f\bigl(\tfrac{\pp}{5}\bigr)\Bigr),
\end{eqnarray*}
connected with Chebyshev and Lobatto quadratures, respectively.	W\k{a}sowicz \cite{SzWas07b}, \cite{SzWas08b} proved that
\begin{equation}\label{eq:w3}
 C(f)\leqslant \Lob{4}(f), \quad if \: f\:is \:3-convex.
\end{equation}
The proof given in \cite{SzWas07b} is rather complicated. This was done using computer software. In \cite{SzWas08b} can be found a new proof, based on the spline approximation of convex functions of higher order.
% (see also Proposition \ref{prop:2.5} of Denuit, Lef\`{e}vre and Shaked   on necessary and sufficient conditions for higher order convex ordering). 
Using Corollary \ref{cor:2.7b} we give a new much simpler proof of \eqref{eq:w3}.

Since for the random variables $X$ and $Y$ with the distributions
\begin{eqnarray*}
\mu_X & = &\tfrac{1}{3}\Bigl( \delta_{-\frac{\sqrt{2}}{2}} + \delta_0+\delta_\frac{\sqrt{2}}{2}\Bigr), \\ \\
\mu_{Y} & = & \tfrac{1}{12}\Bigl(\delta_{-1}+\delta_1\Bigr)+\tfrac{5}{12}\Bigl(\delta_{-\frac{\sqrt{5}}{5}} + \delta_\frac{\sqrt{5}}{5}\Bigr),
\end{eqnarray*}
respectively, we have
$$
C(f) = E[f(X)],\quad \Lob{4}(f) = E[f(Y)], 
$$
it follows that the inequality \eqref{eq:w3} can be rewritten in terms of higher order convex orderings
\begin{equation}\label{eq:w4}
 X \leqslant_{4-cx} Y.
\end{equation}
It is worth noting, that Proposition \ref{prop:2.5} of  Denuit, Lef\`{e}vre and Shaked does not apply to proving \eqref{eq:w4}, because of the distribution functions   $F_X$ and $F_Y$ cross exactly 5-times. We prove the inequality \eqref{eq:w4} by using Corollary \ref{cor:2.7b}.

We have $F_1=F_X$, $F_2=F_Y$, $H_0=F=F_Y-F_X$. By \eqref{eq:2.8c} and \eqref{eq:2.8d}, we obtain
\begin{eqnarray*}
H_3(x)&=&\tfrac{1}{72}\left\{ \left( -1-x \right)^3_+ +\left( 1-x \right)^3_+ 
+5\left[\left( -\tfrac{\sqrt{5}}{5}-x \right)^3_+ +\left( \tfrac{\sqrt{5}}{5}-x \right)^3_+\right] \right.\\ \\
\quad &\,& \left.-4\left[\left( -1-x \right)^3_+ +\left( -\tfrac{\sqrt{2}}{2}-x \right)^3_+ +\left( -x \right)^3_+ +\left( \tfrac{\sqrt{2}}{2}-x \right)^3_+ \right] \right\},
\end{eqnarray*}
\begin{eqnarray*}
H_2(x)&=&\tfrac{1}{24}\left\{ -\left( -1-x \right)^2_+ -\left( 1-x \right)^2_+ 
-5\left[\left( -\tfrac{\sqrt{5}}{5}-x \right)^2_+ +\left( \tfrac{\sqrt{5}}{5}-x \right)^2_+\right] \right.\\ \\
\quad &\,& \left.+4\left[\left( -1-x \right)^2_+ +\left( -\tfrac{\sqrt{2}}{2}-x \right)^2_+ +\left( -x \right)^2_+ +\left( \tfrac{\sqrt{2}}{2}-x \right)^2_+ \right] \right\}.
\end{eqnarray*}
Similarly, from the equality $H_1(x)=H_2^{\,'}(x)$ can be obtained $H_1(x)$. We compute that $x_1=-1-\sqrt{5}+2 \sqrt{2}$, $x_2=0$, $x_3=1+\sqrt{5}-2 \sqrt{2}$ are the points of sign changes of the function $H_2(x)$. It is not difficult to check that the assumptions of Corollary \ref{cor:2.7b} are satisfied.  Since
$$  
(-1)^{3+1}H_3(x_2)=(-1)^{3+1}H_3(0)= \tfrac{1}{72}+\tfrac{\sqrt{5}}{360}-\tfrac{\sqrt{2}}{72}>0,
$$
it follows that the inequalities \eqref{eq:2.7d} are satisfied. From Corollary \ref{cor:2.7b} we conclude that the relation \eqref{eq:w4} hold.



%% References


\bibliographystyle{elsarticle-num}

\begin{thebibliography}{00}
\bibitem{Bes08}
M. BESSENYEI, \textit{Hermite--Hadamard-type inequalities for generalized convex functions}, J. Inequal. Pure Appl. Math., \textbf{9} (2008), 1--51.

\bibitem{BesPal02}
M. BESSENYEI AND ZS. P\'{A}LES,\textit{ Higher-order generalizations of Hadamard's inequality}, Publ. Math. Debrecen, \textbf{61} (2002), no. 3-4, 623--643.

\bibitem{BesPal03}
M. BESSENYEI AND ZS. P\'{A}LES, \textit{Hadamard-type inequalities for generalized convex functions}, Math. Inequal. Appl., \textbf{6}, 3 (2003), 379--392.

\bibitem{BesPal04}
M. BESSENYEI AND ZS. P\'{A}LES, \textit{On generalized higher-order convexity and Hermite--Hadamard-type inequalities}, Acta Sci. Math. (Szeged), \textbf{70}
(2004), no. 1-2, 13--24. MR 2005e:26012.

\bibitem{BesPal10}
M. BESSENYEI AND ZS. P\'{A}LES, \textit{Characterization of higher-order monotonicity via integral inequalities}, Proc. R. Soc. Edinburgh Sect. A, \textbf{140A}, 1 (2010), 723--736.
%\bibitem{Billingsley1986}
%P. Billingsley, Probability and Measure,  New York:  Wiley,  1986.
\bibitem{BraPet}
H. BRASS AND K. PETRAS, \textit{Quadrature theory. The theory of numerical integration on a compact interval},
Mathematical Surveys and Monographs, \textbf{178}. American Mathematical Society, Providence, RI, 2011.

\bibitem{BraSch81}
 H. BRASS AND G. SCHMEISSER, \textit{Error estimates for interpolatory quadrature formulae}, Numer. Math., \textbf{37}, 3 (1981), 371--386.

\bibitem{DenLefSha98}
M. DENUIT, C.LEF\`{E}VRE AND  M. SHAKED, \textit{The s-convex orders among real random
variables, with applications}, Math. Inequal. Appl., \textbf{1} (1998),  585--613.


\bibitem{DraPe2000}
S. S. DRAGOMIR AND C. E. M. PEARCE,
\textit{Selected Topics on Hermite-Hadamard Inequalities and Applications},
RGMIA Monographs, Victoria University, 2000. (Online: http://rgmia.vu.edu.au/monographs/).

\bibitem{Florea}
A. FLOREA, E. P\u{A}LT\u{A}NEA AND D. B\u{A}L\u{A},
\textit{Convex Ordering Properties and Applications}, J. Math. Inequal, \textbf{9}, 4 (2015), 1245--1257. 

\bibitem{Kuczma1985}
M. KUCZMA, \textit{An Introduction to the Theory of Functional Equations and Inequalities}, Prace Naukowe Uniwersytetu \'{S}l\k{a}skiego w Katowicach, vol. \textbf{489}, Pa\'{n}stwowe Wydawnictwo Naukowe -- Uniwersytet \'{S}l\k{a}ski, Warszawa, Krak\'{o}w, Katowice, 1985.

\bibitem{LevinSteckin1960}
V.I. LEVIN AND S.B. STE\v{C}KIN,
\textit{Inequalities.}, Amer. Math. Soc. Transl., \textbf{14}, 2 (1960), 1--22.

\bibitem{NicPer06}
C. P. NICULESCU AND L. E. PERSSON,
\textit{Convex functions and their applications. A contemporary approach},
Springer, New York 2006.

\bibitem{OlbrysSzostok2014}
A. OLBRY\'{S}, T. SZOSTOK,
 \textit{Inequalities of the Hermite-Hadamard type involving numerical differentiation formulas},
Results. Math., \textbf{67} (2015), 403--416.

\bibitem{Ohlin69}
J. OHLIN, \textit{On a class of measures of dispersion with application to optimal reinsurance}, ASTIN Bulletin, \textbf{5} (1969), 249--266.


\bibitem{Popoviciu1934}
T. POPOVICIU, \textit{Sur quelques proprietes des fonctions d'une ou de deux variables reelles}, Mathematica, \textbf{8} (1934), 1--85.


\bibitem{Rajba14}
%\bibitem{Rajba2013}
T. RAJBA, \textit{On probabilistic characterizations of convexity and delta-convexity}, Conference on Inequalities and Applications '14, September 7-13, 2014, Hajd\'uszoboszl\'o (Hungary)


\bibitem{Rajba2011}
T. RAJBA, \textit{New integral representations of $n$th order convex functions}, J. Math. Anal. Appl., \textbf{379}, 2 (2011), 736--747.


\bibitem{Rajba12f}
%\bibitem{Rajba2013}
T. RAJBA, \textit{On the Ohlin lemma for Hermite-Hadamard-Fej\'er type inequalities}, Math. Inequal. Appl., \textbf{17}, 2 (2014), 557--571.



\bibitem{Rajba12g}
T.~RAJBA,
\textit{On strong delta-convexity and Hermite-Hadamard type inequalities for delta-convex functions of higher order}, Math. Inequal. Appl., \textbf{18}, 1 (2015), 267--293. 

\bibitem{TSzostok2014}
T. SZOSTOK, \textit{Levin Ste\v{c}kin theorem and inequalities of the Hermite-Hadamard type}, arXiv preprint, arXiv:1411.7708v1 [math.CA].

\bibitem{Szostok2014a}
T. SZOSTOK,
 \textit{Ohlin's lemma and some inequalities of the Hermite-Hadamard type},
 Aequationes mathematicae, (2014), 1--12 , August 06, 2014.


\bibitem{SzWas07b}
SZ. W\k{A}SOWICZ, \textit{Inequalities between the quadrature operators and error bounds of quadrature rules}, J. Inequal. Pure Appl. Math., \textbf{8}, 2 (2007), Article 42, 8 pp.

\bibitem{SzWas08}
SZ. W\k{A}SOWICZ, \textit{On quadrature rules, inequalities and error bounds},
J. Inequal. Pure Appl. Math. \textbf{9}, 2 (2008), Article 36, 4 pp.

\bibitem{SzWas08b}
SZ. W\k{A}SOWICZ, \textit{A new proof of some inequality connected with quadratures},
J. Inequal. Pure Appl. Math., \textbf{9}, 1 (2008), Article 7, 3 pp.

\bibitem{SzWas10}
SZ. W\k{A}SOWICZ, \textit{On some extremalities in the approximate integration}, Math. Inequal. Appl., \textbf{13} (2010), 165--174.
%%%%%%%%%%%%%%%%%%%%%%%%%%%%%%%%%%%%%%%%%%%%%%%%%%%%%%%%%%%%%




\end{thebibliography}



\end{document}